\documentclass[12pt]{amsart}
\usepackage{fullpage}
\usepackage{amssymb, amsmath, amsfonts, amsthm}
\usepackage{dsfont, mathtools, mathrsfs, enumitem}
\usepackage{cite, thmtools, thm-restate}
\usepackage[pdftex, colorlinks]{hyperref}
\usepackage{cleveref}
\usepackage[T1]{fontenc}

\newtheorem{thm} {Theorem}[section]

\newtheorem{lem} [thm]{Lemma}
\newtheorem{cor} [thm]{Corollary}
\newtheorem{defi}[thm]{Definition}
 
\newtheorem{qst} [thm]{Question} 
 
\newcommand{\rvline}[1]{\hspace*{#1}\vline\hspace*{#1}}
\newcommand{\var}[2][]{V_{#1}({#2})}
\newcommand{\monom}[2]{\binom{#1 + #2}{#2}}
\newcommand{\tor}{\mathrm{tor}}
\newcommand{\red}{\mathrm{red}}
\newcommand{\et}{\mathrm{\acute{e}t}}
\newcommand{\sing}{\mathrm{sing}}
\newcommand{\ab}{\mathrm{ab}}

\newcommand{\an}{\mathrm{an}}

\DeclareMathOperator{\codim}{codim}

\DeclareMathOperator{\HP}{HP}

\DeclareMathOperator{\Proj}{Proj}
\DeclareMathOperator{\Spec}{Spec}
\DeclareMathOperator{\rank}{rank}
\DeclareMathOperator{\conn}{conn}
\DeclareMathOperator{\irr}{irr}
\DeclareMathOperator{\NS}{NS}
\DeclareMathOperator{\Gr}{\mathbf{Gr}}
\DeclareMathOperator{\Pic}{\mathbf{Pic}}
\DeclareMathOperator{\Alb}{\mathbf{Alb}}

\DeclareMathOperator{\Hilb}{\mathbf{Hilb}}
\DeclareMathOperator{\EffDiv}{\mathbf{CDiv}}
\DeclareMathOperator{\IrrDiv}{\mathbf{PDiv}}
\DeclareMathOperator{\Chow}{\mathbf{Chow}}

\DeclareMathOperator{\charac}{char}

\begin{document}
\author{Hyuk Jun Kweon}
\address{Department of Mathematics, Massachusetts Institute of Technology, Cambridge, MA
02139-4307, USA}
\email{kweon@mit.edu}
\urladdr{https://kweon7182.github.io/}
\date{\today}
\subjclass[2010]{Primary 14C05; Secondary 14C20, 14C22} 
\keywords{N\'eron--Severi group, Castelnuovo–Mumford regularity, Gotzmann number}
\thanks{This research was partially supported by Samsung Scholarship and National Science Foundation grant DMS-1601946.}
\title{Bounds on the Torsion Subgroups of N\'eron--Severi Groups}

\begin{abstract}
  Let $X \hookrightarrow \mathbb{P}^r$ be a smooth projective variety defined by homogeneous polynomials of degree $\leq d$. We give explicit upper bounds on the order of the torsion subgroup $(\NS X)_{\tor}$ of the N\'eron--Severi group of $X$. The bounds are derived from an explicit upper bound on the number of irreducible components of either the Hilbert scheme $\Hilb_Q X$ or the scheme $\EffDiv_n X$ parametrizing the effective Cartier divisors of degree $n$ on $X$. We also give an upper bound on the number of generators of $(\NS X)[\ell^\infty]$ uniform as $\ell\neq\charac k$ varies.
\end{abstract}
\maketitle

\section{Introduction}
The N\'eron--Severi group $\NS X$ of a smooth proper variety $X$ over a field $k$ is the group of divisors modulo algebraic equivalence. If $k$ is algebraically closed, $\NS X$ also equals the group of connected components of the Picard scheme of $X$. N\'eron \cite[p. 145, Th\'eor\`eme 2]{Ner} and Severi \cite{Sev} proved that $\NS X$ is finitely generated. The aim of this paper is to give an explicit upper bound on the order of $(\NS X)_{\tor}$. To the best of the author's knowledge, this is the first explicit bound on the order of $(\NS X)_{\tor}$.
\begin{restatable*}{thm}{NSbound}\label{thm:NS better bound}
  Let $X \hookrightarrow \mathbb{P}^r$ be a smooth projective variety defined by homogeneous polynomials of degree $\leq d$. Then
  \begin{equation}
    \#(\NS X)_\tor \leq  2^{d^{r^2+2r\log_2 r}}.
  \end{equation}
\end{restatable*}

\noindent For any prime number $\ell \neq \charac k$, there is a natural isomorphism $(\NS X)[\ell^\infty] \simeq H^2_\et(X,\mathbb{Z}_\ell)_\tor$ \cite[2.2]{SZ}. Thus, \Cref{thm:NS better bound} implies the corollary below.

\begin{restatable*}{cor}{SHbound}\label{thm:H2 bound}
  Let $X \hookrightarrow \mathbb{P}^r$ be a smooth projective variety defined by homogeneous polynomials of degree $\leq d$. Let $\ell \neq \charac k$ be a prime number. Then
  \begin{equation*}
    \prod_{\substack{\ell \neq \charac k\\ \ell \textup{ is prime}}} \# H^2_\et(X,\mathbb{Z}_\ell)_\tor \leq 2^{d^{r^2+2r\log_2 r}}.
  \end{equation*}
\end{restatable*}
\noindent We also give a uniform upper bound on the number of generators of $(\NS X)[\ell^\infty]$ as described below. This bound and a sketch of its proof were suggested by J\'anos Koll\'ar after seeing an earlier draft of our paper containing only \Cref{thm:NS better bound}.

\begin{restatable*}{thm}{NSgen}\label{thm:NSgen}
  Let $X \hookrightarrow \mathbb{P}^r$ be a smooth connected projective variety of degree $d$ over $k$. Let $p=\charac k$, and
  \[N = \begin{cases}
      (\NS X)_\tor, & \mbox{if } p=0 \\
      (\NS X)_\tor/(\NS X)[p^\infty], & \mbox{if } p>0.
    \end{cases}\]
  Then $N$ is generated by less than or equal to $(d-1)(d-2)$ elements.
\end{restatable*}

\begin{restatable*}{cor}{SHgen}
  Let $X \hookrightarrow \mathbb{P}^r$ be a smooth connected projective variety of degree $d$ over $k$. Let $\ell \neq \charac k$ be a prime number. Then $H^2_\et(X,\mathbb{Z}_\ell)_\tor$ is generated by less than or equal to $(d-1)(d-2)$ elements.
\end{restatable*}


The torsion-free quotient $(\NS X)/(\NS X)_\tor$ is the group of divisors modulo numerical equivalence. Its rank is bounded by the second Betti number of $X$. Katz found an upper bound on the sum of all Betti numbers of $X$ \cite[Theorem 1]{Kat}; this gives a rough bound on the rank of $\NS X$.

The torsion subgroup $(\NS X)_{\tor}$ is the group of numerically zero divisors modulo algebraic equivalence. This is a birational invariant \cite[p. 177]{Bal}. Recently, Poonen, Testa and van Luijk gave an algorithm to compute it.\footnote{They also found an algorithm computing the torsion free quotient assuming the Tate conjecture.} The algorithm is based on their theorem that $(\NS X)_{\tor}$ injects into the set of connected components of $\Hilb_Q X$ for some polynomial $Q$ \cite[Lemma 8.29]{PTL}. The order of $(\NS X)_{\tor}$ will be bounded by finding an explicit $Q$ and bounding the number of connected components of $\Hilb_Q X$. We will show that an upper bound on the number of connected components of $\EffDiv_n X$ for an integer $n$ depending on $Q$ also gives an upper bound on the order of $(\NS X)_{\tor}$.

\Cref{sec:numerical} shows that $Q$ may be taken to be the Hilbert polynomial of $mH$, where $H$ is a hyperplane section of $X$ and $m$ is an explicit integer.
\Cref{sec:hilbert} bounds the number of irreducible components of $\Hilb_Q X$ by using its embedding in a Grassmannian. \Cref{sec:chow} bounds the number of irreducible component of $\EffDiv_n X$ by using Koll\'ar's technique \cite[Exercise I.3.28]{Jan}. Finally, \Cref{sec:generators} gives a uniform upper bound on the number of generators of $(\NS X)[\ell^\infty]$.

From now on, the base field $k$ is assumed to be algebraically closed, since a base change makes the N\'eron--Severi group only larger \cite[Proposition 6.1]{PTL}. However, no assumption is made on the characteristic of $k$.

\section{Notation}
Given a scheme $X$ over $k$, let $\conn(X)$ and $\irr(X)$ be the set of connected components of $X$ and the set of irreducible components of $X$, respective. Let $X_\red$ be the reduced closed subscheme associated to $X$. If $X$ is smooth and proper, then $\NS X$ denote the N\'eron--Severi group of $X$. Let $H^i_\et(X,\mathscr{F})$ be the $i$-th \'etale cohomology group of $X$ corresponding to an \'etale sheaf $\mathscr{F}$. If $k = \mathbb{C}$, then let $X^\an$ be the analytic space of $X$. If $M$ is a topological manifold, then $H_\sing^i(M,\mathbb{Z})$ (resp. $H_i^\sing(M,\mathbb{Z})$) denotes the $i$-th singular cohomology (resp. homology) group with integer coefficients.

A projective variety is a closed subscheme of $\mathbb{P}^r=\Proj k[x_0,\cdots,x_r]$ for some $r$. Suppose that $X \hookrightarrow \mathbb{P}^r$ is a projective variety. Let $I_X \subset k[x_0,\cdots,x_r]$ be the saturated ideal defining $X$. Given $f_0,\cdots,f_{t-1} \in k[x_0,\cdots,x_r]$, let $\var[X]{f_0,\cdots,f_{t-1}}$ be the subscheme of $X$ defined by the ideal $(f_0,\cdots,f_{t-1})+I_X$. Let $\mathcal{O}_X$, $\mathscr{I}_X$, $\Omega_X$ and $\omega_X$ be the sheaf of regular functions, the ideal sheaf, the sheaf of differentials and the canonical sheaf of $X$, respectively.

Given a coherent sheaf $\mathscr{F}$ on $X$, let $\HP_\mathscr{F}$ be the Hilbert polynomial of $\mathscr{F}$, and $\Gamma(\mathscr{F})$ be the global section of $\mathscr{F}$. Given a graded module $M$ over $k$, let $\HP_M$ be the Hilbert polynomial of $M$, and $M_t$ be the degree $t$ part of $M$. Take an effective divisor $D$ on $X$. Let $\HP_D$ be the Hilbert polynomial of $D$ as a subscheme of $\mathbb{P}^r$. Then $\mathcal{O}(-D)\subset \mathcal{O}_X$ is the ideal sheaf corresponding to $D$, and
\[ \HP_D = \HP_{\mathcal{O}_X} - \HP_{\mathcal{O}(-D)}.\]

Let $\Hilb X$ be the Hilbert scheme of $X$. Given a polynomial $Q(t)$, let $\Hilb_Q X$ be the Hilbert scheme of $X$ parametrizing closed subschemes of $X$ with Hilbert polynomial $Q$. Let $\Chow_{\delta,n} X$ be the Chow variety of dimension $\delta$ and degree $n$ algebraic cycles on $X$. Let $\EffDiv X$ be the scheme parametrizing the effective Cartier divisors on $X$, and let $\EffDiv_n X$ be the open and closed subscheme of $\EffDiv X$ corresponding to the divisors of degree $n$. Let $\Pic X$ be the Picard scheme of $X$. Let $\Alb X$ be the Albanese variety of $X$. Given a vector space $V$ and nonnegative number $t$, let $\Gr(t,V)$ be the Grassmannian parametrizing $t$-dimensional subspaces of $V$. Let $\Gr(t,n) = \Gr(t,k^n)$. 

Given a set $S$, let $\# S$ be the number of elements in $S$. Give a group $A$, let $A^\ab$ be the abelianization of $A$, and $A^{(\ell)} = \varprojlim A^\ab/\ell^n A^\ab$ be the maximal pro-$\ell$ abelian quotient of $A$. If $A$ is abelian, let $A_\tor$, $A[n]$ and $A[\ell^\infty]$ be the set of torsion elements, $n$-torsion elements and $\ell$-power torsion elements, respectively. If $A$ is a finite abelian group, then $A^*$ is the Pontryagin dual of $A$.




\section{Numerical Conditions}\label{sec:numerical}
Let $X \hookrightarrow \mathbb{P}^r$ be a smooth projective variety defined by polynomials of degree $\leq d$. Let $K$ and $H$ be a canonical divisor and a hyperplane section of $X$, respectively. The goal of this section is to give an explicit $m$ such that
\begin{equation}\label{eqn:sec2goal}
  \# \conn(\Hilb_{\HP_{mH}}X) \geq \# (\NS X)_\tor.
\end{equation}

Poonen, Testa and van Luijk proved that an order of a N\'eron--Severi group is bounded by a number of connected components of a certain Hilbert scheme. The theorem and the proof below is a reformulation of their work in \cite[Section 8.4]{PTL}.

\begin{thm}[Poonen, Testa and van Luijk]\label{thm:bound by Hilb}
  Let $F$ be a divisor on $X$ and $Q = \HP_{\mathcal{O}_X} - \HP_{\mathcal{O}(-F)}$. If $\mathcal{O}_X(F+D)$ has a global section for every numerically zero divisor $D$, then
  \[ \# \conn(\Hilb_Q X \cap \EffDiv X) \geq \# (\NS X)_\tor.\]
\end{thm}
\begin{proof}
  Recall that $\EffDiv X$ is an open and closed subscheme of $\Hilb X$ \cite[Exercise I.3.28]{Jan}, and there is a natural proper morphism $\pi: \EffDiv X \rightarrow \Pic X$ sending a divisor to the corresponding class \cite[p. 214]{BLR}.

  Let $\Pic^c X$ be the finite union of connected components of $\Pic X$ parametrizing the divisors numerically equivalent to $F$. Since the Hilbert polynomial of a divisor is a numerical invariant, $Q$ is the Hilbert polynomial of each divisor corresponding to a closed point of $\pi^{-1}(\Pic^c X)$. Thus, $\pi^{-1}(\Pic^c X) \subset \Hilb_Q X$. Since $\Pic^c X$ is open and closed in $\Pic X$, and $\EffDiv X$ is open and closed in $\Hilb X$, the scheme $\pi^{-1}(\Pic^c X)$ is open and closed in $\Hilb_Q X \cap \EffDiv X$. Thus,
  \begin{align*}
    \# \conn(\Hilb_Q X \cap \EffDiv X) &\geq \# \conn(\pi^{-1}(\Pic^c X)).
  \end{align*}
  Since $F+D$ is linearly equivalent to an effective divisor for every numerically zero divisor $D$, the morphism $\pi$ restricts to a surjection $\pi^{-1}(\Pic^c X) \rightarrow \Pic^c X$. Hence,
  \begin{align*}
    \# \conn(\pi^{-1}(\Pic^c X)) &\geq \# \conn(\Pic^c X) \\
                                 &= (\NS X)_\tor. \qedhere
  \end{align*}
\end{proof}

The authors of \cite{PTL} chose $F = K + (\dim X + 2)H$ because of the following partial result towards Fujita's conjecture.

\begin{thm}[Keeler {\cite[Theorem 1.1]{Kee}}]\label{thm:Fujita}
  Let $L$ be an ample divisor on X. Then
  \begin{enumerate}[label=(\alph*)]
  \item $\mathcal{O}_X(K+(\dim X)H+L)$ is generated by global sections, and
  \item $\mathcal{O}_X(K+(\dim X + 1)H+L)$ is very ample.
  \end{enumerate}
\end{thm}

However, computing the Hilbert polynomial of $\mathcal{O}(-K)$ is somehow difficult. Therefore, we will show that $F = ((d-1)\cdot\codim X)H$ is another choice.

\begin{lem}\label{lem:section of det}
  Let $Y$ be a nonempty smooth closed subscheme of an affine space $\mathbb{A}^r = \Spec k[x_0,\cdots,x_{r-1}]$. Suppose that the ideal $I$ defining $Y$ is generated by polynomials of degree $\leq d$. Let $c = \codim Y$. Then there are polynomials $f_0, \cdots, f_{c-1} \in I$ such that
  \begin{enumerate}[label=(\alph*)]
  \item $\deg f_i = d$ for all $i$, and
  \item $f_0\wedge\cdots\wedge f_{c-1}$ represents a nonzero element in
    $\bigwedge^c(I/I^2)$.
  \end{enumerate}
\end{lem}
\begin{proof}
  Let $R = k[x_0,\cdots,x_{c-1}]/I$. Then $I/I^2$ is a locally free $R$-module of rank $\codim Y$ by \cite[Theorem 8.17]{Har}. Take any prime ideal $\mathfrak{p} \subset R$. Then $(I/I^2)_\mathfrak{p}$ is a free $R_\mathfrak{p}$-module. Thus, there are $p_0,\cdots,p_{c-1} \in I$ and $q_0,\cdots,q_{c-1} \in R \setminus \mathfrak{p}$ such that
  \[\frac{p_0}{q_0} \wedge \frac{p_1}{q_1} \wedge \cdots \wedge \frac{p_{c-1}}{q_{c-1}} \]
  represents a nonzero element of $\bigwedge^c(I/I^2)_\mathfrak{p}$. Hence,
  \begin{equation}\label{eqn:wedge}
    p_0 \wedge p_1 \wedge \cdots \wedge p_{c-1}
  \end{equation}
  represents a nonzero element of $\bigwedge^c(I/I^2)$.

  Let $g_0,\cdots,g_{b-1}$ be polynomials of degree $\leq d$ which generate $I$. Then each $p_i$ can be written as a $R$-linear combination of $g_i$'s. If we expand (\ref{eqn:wedge}), at least one term should be nonzero in $\bigwedge^c(I/I^2)$. Therefore, we may assume that
  \[g_0 \wedge g_1 \wedge \cdots \wedge g_{c-1} \]
  represents a nonzero element of $\bigwedge^c(I/I^2)$. Let $\ell \not\in I$ be a polynomial of degree 1, and let $f_i = \ell^{d - \deg g_i} g_i$. Then
  \[f_0 \wedge f_1 \wedge \cdots \wedge f_{c-1} \]
  represents a nonzero element of $\bigwedge^c(I/I^2)$ and $\deg f_i = d$ for every $i$.
\end{proof}

\begin{lem}\label{lem:section anticanonical}
  The sheaf $\mathcal{O}_X\left(-K+(d \cdot \codim X - r - 1)H\right)$ has a global section.
\end{lem}
\begin{proof}
  Let $c = \codim X$. Since $X$ is smooth, there is an exact sequence
  \[ 0 \rightarrow \mathscr{I}_X/\mathscr{I}_X^2 \rightarrow
    \Omega_{\mathbb{P}^r} \otimes \mathcal{O}_X \rightarrow \Omega_X \rightarrow 0, \]
  and $\mathscr{I}_X/\mathscr{I}_X^2$ is locally free of rank $c$ \cite[Theorem 8.17]{Har}. Taking the highest exterior power gives
  \begin{align*}
    \omega_{\mathbb{P}^r}|_X &\simeq
    \bigwedge\nolimits^c \left(\mathscr{I}_X/\mathscr{I}_X^2\right) \otimes \omega_X\\
    \omega_X^{-1}(-r-1) &\simeq \bigwedge\nolimits^c \left(\mathscr{I}_X/\mathscr{I}_X^2\right).
  \end{align*}
  Let $U_i\subset X$ be the affine open set given by $x_i \neq 0$. Then there exist polynomials $f_0,\cdots,f_{c-1}$ of degree $d$ such that
  \[ f_0(x_0/x_i,\cdots,x_r/x_i) \wedge \cdots \wedge f_{c-1}(x_0/x_i,\cdots,x_r/x_i) \]
  represents a nonzero section of $\bigwedge^c(\mathscr{I}_X/\mathscr{I}_X^2)|_{U_i}$ by \Cref{lem:section of det}. Take another $U_j \subset X$ given by $x_j \neq 0$. Then two sections
  \[x_i^d f_0(x_0/x_i,\cdots,x_r/x_0) \wedge \cdots \wedge
    x_i^d f_{c-1}(x_0/x_i,\cdots,x_r/x_i)\]
  and
  \[x_j^d f_0(x_0/x_j,\cdots,x_r/x_j) \wedge \cdots \wedge
    x_j^d f_{c-1}(x_0/x_j,\cdots,x_r/x_j)\]
  give same restrictions in $\bigwedge^c(\mathscr{I}_X/\mathscr{I}_X^2(d))|_{U_i \cap U_j}$. Because $i$ and $j$ are arbitrary, the sections above extend to a global section of
  \[ \omega_X^{-1}(d\cdot\codim X-r-1) \simeq
    \bigwedge\nolimits^c \left(\mathscr{I}_X/\mathscr{I}_X^2(d)\right). \qedhere \]
\end{proof}

\begin{lem}\label{lem:numerical condition}
  Let $D$ be a divisor on $X$ numerically equivalent to 0.\footnote{The condition `numerically equivalent to 0' can be replaced by `numerically effective' due to Kleiman's criterion of ampleness \cite[Chapter IV \textsection 2 Theorem 2]{Kle}.} Then
  \begin{enumerate}[label=(\alph*)]
  \item $\mathcal{O}_X(D+((d-1)\codim X)H)$ is generated by global sections, and
  \item $\mathcal{O}_X(D+((d-1)\codim X+1)H)$ is very ample.
  \end{enumerate}
\end{lem}
\begin{proof}
  The divisor $D+H$ is ample, since ampleness is a numerical property. Then $K+(\dim X)H + (D+H)$ is generated by global section by \Cref{thm:Fujita}. Thus, \Cref{lem:section anticanonical} implies that
  \begin{align*}
    &(K+(\dim X)H + (D+H)) + (-K+(d \cdot \codim X - r - 1)H) \\
    =\, &D+((d-1)\codim X)H
  \end{align*}
  is generated by global sections. Similarly, $D+((d-1)\codim X+1)H$ is very ample.
\end{proof}


\begin{thm}\label{cor:bound by hilb}
  Let $m = (d-1) \codim X$. Then
  \[\# \conn(\Hilb_{\HP_{m H}} X) \geq \# (\NS X)_\tor.\]
\end{thm}

\begin{proof}
  By \Cref{lem:numerical condition}(a), we may apply \Cref{thm:bound by Hilb} to $F=mH$.
\end{proof}

\section{Irreducible Components of Hilbert Schemes}\label{sec:hilbert}
The aim of this section is to give an explicit upper bound on $\#(\NS X)_{\tor}$ for a smooth projective variety $X$. \Cref{cor:bound by hilb} implies that it suffices to give an upper bound on the number of connected components of some Hilbert scheme. Recall the definition of Castelnuovo–Mumford regularity and Gotzmann numbers.

\begin{defi}
  A coherent sheaf $\mathscr{F}$ over $\mathbb{P}^r$ is $m$-regular if and only if
  \[ H^i\left(\mathbb{P}^r, \mathscr{F}(m-i)\right) = 0 \]
  for every integer $i > 0$. The smallest such $m$ is called the Castelnuovo–Mumford regularity of $\mathscr{F}$.
\end{defi}

\begin{defi}
  Let $P$ be the Hilbert polynomial of some ideal $I \subset k[x_0,\cdots,x_r]$. The Gotzmann number $\varphi(P)$ of $P$ is defined as 
  \begin{align*}
    \varphi(P) = \inf \{ m \,|\,& \mathscr{I}_Z  \text{ is $m$-regular for every} \\
        &\text{closed subvariety $Z \subset \mathbb{P}^r$ with Hilbert polynomial $P$} \}.
  \end{align*}
\end{defi}

Hilbert schemes can be explicitly described as a closed subscheme of a Grassmannian, by Gotzmann \cite{Got}.


\begin{thm}[{Gotzmann}]\label{thm:explicit construction of Hilb P}
  Let $P$ be the Hilbert polynomial of some ideal $I \subset k[x_0,\cdots,x_r]$. Assume that $t \geq \varphi(P)$. Then
  \begin{align*}
    \imath_t : \Hilb_{\monom{t}{r}-P(t)} \mathbb{P}^r &\rightarrow \Gr(P(t),k[x_0,\cdots,x_r]_t) \\
    [Y] &\mapsto \Gamma(\mathscr{I}_Y(t))
  \end{align*}
  gives a well-defined closed immersion. Moerover, the image is the collection of linear spaces $T \subset k[x_0,\cdots,x_r]_t$ such that
  \begin{enumerate}[label=(\alph*)]
  \item $\dim \left(x_0T+\cdots+x_rT\right) \leq P(t+1)$.
  \end{enumerate}
\end{thm}

\begin{proof}
  See \cite[Section 3]{Got}.
\end{proof}

Let $X \hookrightarrow \mathbb{P}^r$ be a projective variety and $Q$ be a polynomial. Then there is the natural closed embedding
  \[ \Hilb_Q X \hookrightarrow \Hilb_Q \mathbb{P}^r. \]

\begin{thm}\label{thm:explicit construction}
  Use the notation in \Cref{thm:explicit construction of Hilb P}. Let $X \hookrightarrow \mathbb{P}^r$ be a projective variety defined by polynomials of degree $\leq d$. Assume that $t \geq \max \{\varphi(P),d\}$. Then the image of
  \[\Hilb_{\monom{t}{r}-P(t)} X \]
  under $\imath_t$ is the collection of linear spaces $T \subset k[x_0,\cdots,x_r]_t$ such that
  \begin{enumerate}[label=(\alph*)]
  \item $\dim \left(x_0T+\cdots+x_rT\right) \leq P(t+1)$ and
  \item $\Gamma(\mathscr{I}_X(t)) \subset T$.
  \end{enumerate}
\end{thm}

\begin{proof}
  See the proof of \cite[Lemma 8.23]{PTL} 
\end{proof}



Therefore, an upper bound on Gotzmann numbers will give an explicit construction of a Hilbert scheme. Such a bound is given by Hoa {\cite[Theorem 6.4(i)]{Hoa}}.
 
\begin{thm}[Hoa]\label{thm:Gotzmann bound}
  Let $I\subset k[x_0,\cdots,x_r]$ be an nonzero ideal generated by homogeneous polynomials of degree at most $d \geq 2$. Let $a$ be the Krull dimension of  $k[x_0,\cdots,x_r]/I$. Then
  \[ \varphi(\HP_I) \leq \left( \frac{3}{2} d^{r+1-a} + d \right)^{a 2^{a-1}}. \]
\end{thm}

Once a Hilbert scheme is explicitly constructed, we can bound the number of the irreducible components by the lemma below.

\begin{lem}\label{cor:conn}
  If $X \hookrightarrow \mathbb{A}^r$ is an affine scheme defined by polynomials of degree $\leq d$, then
  \[\# \irr(X) \leq d^r. \]
\end{lem}
\begin{proof}
  This is a special case of the Andreotti-Bézout inequality \cite[Lemma 1.28]{Cat}.
\end{proof}


The Grassmannian $\Gr(k,n)$ is covered by open sets isomorphic to $\mathbb{A}^{k(n-k)}$. The conditions in \Cref{thm:explicit construction} can be translated into explicit equations in such an affine space, giving the lemma below.

\begin{lem}\label{lem:counting components}
  Let $V$ and $W$ be vector spaces. Let $n = \dim V$. Let $\varphi_i: V \rightarrow W$ be a linear map for each $i = 0,\cdots,r$. Let $U \subset V$ be a subspace. Let $\mathbf{X}$ be the collection of $T \in \Gr(q,V)$ satisfying
  \begin{enumerate}[label=(\alph*)]
  \item $\dim (\varphi_0(T)+\cdots+\varphi_r(T)) \leq p$, and
  \item $U \subset T$.
  \end{enumerate}
  Then $\mathbf{X}$ is a closed subscheme of $\Gr(q,V)$, and with $n = \dim V$,
  \[ \#\conn(\mathbf{X}) \leq \max\{p+1,q+1\}^{q(n-q)}.\]
\end{lem}

\begin{proof}
  Given a subspace $S \subset V$ of dimension $n-q$, there is an open set
  \[ U_S = \{T\in\Gr(q,V) \,|\, U \cap T = \{0\} \}. \]
  Choose a basis of $V$ such that $S$ is spanned by the first $n-q$ entries. Then every $T \in U_S$ is uniquely represented as a column space of a block matrix
  \[ N_T = \begin{pmatrix} X_{n-q,q} \\ \hline I_q \end{pmatrix},\]
  where $X_{n-q,q}$ is an $(n-q)\times q$ matrix, and $I_q$ is the $q\times q$ identity matrix.
  The entries of $X_{n-q,q}$ can be regarded as indeterminates, giving an isomorphism $U_S \simeq \mathbb{A}^{(n-q)q}$. Let $\varphi_i^{\oplus q}(N_T)$ be the $(\dim W) \times q$ matrix obtained by applying $\varphi_i$ to every column of $N_T$. Then
  \[\dim \left( \varphi_0(T)+\cdots +\varphi_r(T)\right) \leq p\]
  if and only if every $(p+1)\times(p+1)$ minor of the block matrix
  \[\left(
      \varphi_0^{\oplus k}\begin{pmatrix} X_{n-q,q} \\ I_q \end{pmatrix} \rvline{\arraycolsep}
      \varphi_1^{\oplus k}\begin{pmatrix} X_{n-q,q} \\ I_q \end{pmatrix} \rvline{\arraycolsep}
      \cdots \rvline{\arraycolsep}
      \varphi_r^{\oplus k}\begin{pmatrix} X_{n-q,q} \\ I_q \end{pmatrix} \right)\]
  is zero.

  Let $M_U$ be a matrix whose columns form a basis of $U$. Then $U \subset T$ if and only if every $(q+1)\times(q+1)$ minor of
  
  \[\begin{pmatrix}
      \begin{matrix}X_{n-q,q} \\ I_q \end{matrix} & \rvline{-\arraycolsep} & M_U
  \end{pmatrix}\]
is zero. Thus, $\mathbf{X} \cap U_S$ is defined by the minors in $U_S$, meaning that $\mathbf{X}$ is a closed subscheme of $\Gr(q,V)$.

Now, take one point from each irreducible component of $\mathbf{X}$, and let them be $T_0,\cdots,T_{\ell-1}\subset V$. Then $(n-k)$-dimensional subspace $S \subset V$ can be chosen such that
\[ S \cap \left(\bigcup_{i=0}^\ell T_i\right) = \{0\},\]
implying that every irreducible component intersects with $U_S$. Since $\mathbf{X} \cap U_S$ is defined by polynomials of degree $\leq \max\{p+1,q+1\}$, \Cref{cor:conn} implies
\begin{equation*}
  \#\irr(\mathbf{X}) = \#\irr(\mathbf{X}\cap U_S) \leq \max\{p+1,q+1\}^{q(n-q)}. \qedhere
\end{equation*}
\end{proof}

\begin{lem}\label{lem:conn bound}
  Let $X\hookrightarrow\mathbb{P}^r$ be a projective variety defined by polynomials of degree $\leq d$. Let $P$ be the Hilbert polynomial of an ideal. If $t \geq \max\{\varphi(P),d,8r\}$ and $r \geq 2$, then
  \[ \# \irr\left(\Hilb_{\monom{t}{r} - P(t)} X\right) \leq t^{r t^{2r}}. \]
\end{lem}
\begin{proof}
  \Cref{thm:explicit construction} and \Cref{lem:counting components} implies that
  \begin{align*}
    \# \irr\left(\Hilb_{\monom{t}{r} - P(t)} X\right)
    &\leq \max\left\{ P(t+1)+1, P(t)+1 \right\}^{P(t)\left(\monom{t}{r}-P(t)\right)}\\
    &\leq \left( \monom{t+1}{r} + 1 \right)^{\monom{t}{r}^2}.
  \end{align*}
  Since $r\geq 2$ and $t \geq 8r$, 
  \begin{align*}
    \monom{t+1}{r} &= \frac{(t+r+1)\cdots(t+2)}{r!} \\
                   &\leq \frac{(t+r+1)^r}{2^{r-1}} \\
                   &\leq \left(\frac{10}{8}t\right)^r \frac{1}{2^{r-1}} \\
                   &< t^r.
  \end{align*}
  Therefore,
  \[ \#\irr\left(\Hilb_{\monom{t}{r}-P(t)} X\right) \leq \left(t^r\right)^{\left({t^r}\right)^2} = t^{rt^{2r}}.
  \qedhere\]
\end{proof}

Now, we are ready to give an upper bound:
\begin{thm}\label{thm:NS bound}
  Let $X \hookrightarrow \mathbb{P}^r$ be a smooth projective variety defined by homogeneous polynomials of degree $\leq d$. Then
  \begin{equation}
    \# (\NS X)_{\tor} \leq 2^{d^{2^{r+3 \log_2 r}}}.
  \end{equation}
\end{thm}
\begin{proof}
  If $X$ is a curve or a projective space, then $(\NS X)_\tor = 0$. Thus, we may assume that $r \geq 3$, $d \geq 2$ and $\codim X \geq 1$. Moreover, $X$ may be assumed to be not contained in any hyperplane.
  
  Let $I$ be the ideal defining $X$. Let $H$ be a hyperplane section of $X$ cut by $x_r = 0$. Let $m = (d-1)\codim X \geq 1$. Then $I+(x_r^m)$ is the ideal defining $mH$ where $m=(d-1)\codim X$. Let $t = (2rd)^{(r+1)2^{r-2}}$ and $a = \dim X$. Then $t \geq d$ and $t \geq 8r$. Moreover,
  \begin{align*}
    t &= \left( \frac{3}{2} (rd)^{r+1-a} + \frac{1}{2} (rd)^{r+1-a} \right)^{a 2^{r-2}} \\
      &\geq \left( \frac{3}{2} (rd)^{r+1-a} + rd \right)^{a 2^{a-1}} \\
      &\geq \varphi(\HP_{mH}) \text{ (by \Cref{thm:Gotzmann bound})}.
  \end{align*}
  Thus, \Cref{cor:bound by hilb} and \Cref{lem:conn bound} implies that 
  \[ \# (\NS X)_\tor \leq \#\conn(\Hilb_{\HP_{m H}} X) \leq \#\irr(\Hilb_{\HP_{m H}} X) \leq t^{rt^{2r}}. \]
  Notice that
  \begin{align*}
    \log_2 \log_d \log_2 \left( t^{rt^{2r}} \right)
    &= \log_2 \log_d \left( t^{2r} r \log_2t \right) \\
    &\leq \log_2 \log_d \left( t^{2r} 2^{r-1} (r+1)r^2 d \right) \\
    &\phantom{\leq} \text{ (since $\log_2 (2rd) \leq 2rd$)}\\
    &\leq \log_2 \left( 2r \log_dt + r + \log_2 (r+1)r^2 \right) \\
    &\phantom{\leq} \text{ (since $d \geq 2$)}\\
    &\leq \log_2 \left( 2^{r-1}r(r+1)(\log_2 r + 2) + r + \log_2 (r+1)r^2 \right) \\
    &\leq r + 3\log_2 r.
  \end{align*}
  As a result,
  \begin{equation*}
    \# (\NS X)_\tor \leq 2^{d^{2^{r+3 \log_2 r}}}. \qedhere
  \end{equation*}

\end{proof}


\section{Irreducible Components of Chow Varieties}\label{sec:chow}
In this section, $X \hookrightarrow \mathbb{P}^r$ is a smooth projective variety defined by homogeneous polynomials of degree $\leq d$. The goal of this section is to give a better upper bound on the order of $(\NS X)_{\tor}$. J\'anos Koll\'ar pointed out that the bound may be also derived form an upper bound on the number of connected components of $\Chow_{\delta,n} X$. Our goal only requires a bound for $\EffDiv_n X$ instead of Chow varieties of arbitrary dimensions.

\begin{thm}\label{thm:bound by EffDiv}
  Let $n = (d-1) \codim X \cdot \deg X$. Then
  \[\# \conn\left(\EffDiv_n X \right) \geq \# (\NS X)_\tor.\]
\end{thm}
\begin{proof}
  Let $H \subset X$ be a hyperplane section, $m = (d-1) \codim X$, and $Q = \HP_{mH}$. If $D$ is a closed subscheme with Hilbert polynomial $Q$, then $\deg D = \deg mH = n$. Hence,
  \[ \Hilb_Q X \cap \EffDiv X \subset \EffDiv_n X. \]
  Since $\Hilb_Q X$, $\EffDiv X$ and $\EffDiv_n X$ are open and closed in $\Hilb X$, \Cref{lem:numerical condition} and \Cref{thm:bound by Hilb} implies that
  \begin{align*}
    \#\conn(\EffDiv_n X)
    &\geq \#\conn(\Hilb_Q X \cap \EffDiv X) \\
    &\geq \# (\NS X)_\tor. \qedhere
  \end{align*}
\end{proof}

In \cite[Exercise I.3.28]{Jan}, Koll\'ar gives an explicit upper bound on $\# \irr (\Chow_{\delta,n} \mathbb{P}^r)$ and an outline of the proof. Moreover, \cite[Exercise I.3.28.13]{Jan} suggests an exercise to find an explicit upper bound on $\# \irr (\Chow_{\delta,n} X)$, and the proof can be found in \cite[Section 2]{Gue}. However, the proof works only if $\charac k = 0$, and it does not give a bound in a closed form.

Therefore, we will give another complete proof, but most of the proof up to \Cref{lem:IrrDiv bound} is just a modification of Koll\'ar's technique and \cite[Section 2]{Gue}. Moreover, we will only bound $\#\irr(\EffDiv_n X)$, because this restriction avoids the bad behavior of Chow varieties in positive characteristic, simplies the proof and slightly improves the result.

\begin{lem}\label{lem:generic sec}
  Let $D \subset X$ be a nonzero effective divisor on $X$ of degree $n$. Then there are $f$ and $g$ in $\Gamma(X,\mathscr{I}_D(n))$ such that $D$ is the largest effective divisor contained in $V_X(f,g)$.
\end{lem}
\begin{proof}
  Take any point $x \in X$ and a generic linear projection $\rho_0\colon\mathbb{P}^r \dashrightarrow \mathbb{P}^{\dim X}$. Then $\rho_0(D)$ is a hypersurface and $\rho_0|_X$ is \'etale at $x$. Let $f_0$ be a homogeneous polynomial of degree $n$ defining $\rho_0(D)$, and $f \in \Gamma(X,\mathscr{I}_D(n))$ be its pullback. Then
  \[ \var[X]{f} = D \cup E_0 \cup E_1 \cup \cdots \cup E_{t-1} \]
  for some irreducible closed subschemes $E_i \subset X$ not set theoretically contained in $D$.

  Take $e_i \in E_i \setminus D$ for each $i$, and let $\rho_1\colon\mathbb{P}^r \dashrightarrow \mathbb{P}^{\dim X}$ be another generic linear projection. Then $e_i \not\in \rho_1(D)$ for every $i$. Let $g_0$ be a homogeneous polynomial of degree $n$ defining $\rho_1(D)$, and $g \in \Gamma(X,\mathscr{I}_D(n))$ be its pullback. Then $\var[X]{f,g}$ contains $D$ but not $E_i$ for all $i$. Thus, $D$ is the largest divisor contained in $\var[X]{f,g}$.
\end{proof}

\begin{defi}\label{def:S and T}
  Let
  \[ N_n = {n + r \choose n} -1. \]
  Then $\mathbb{P}^{N_n}$ parameterizes nonzero homogeneous polynomials of degree $n$ with $r+1$ variables up to constant factors. Let
  \begin{align*}
    S_n &= \left\{ (f,g) \in \left(\mathbb{P}^{N_n}\right)^2 \ \middle|\ 
        \codim_X (\var[X]{f,g}) \leq 1 \right\} \text{ and}\\
    T_n &= \left\{ ((f,g),[D]) \in S_n \times \EffDiv X \mid
        D \subset \var[X]{f,g} \right\}.
  \end{align*}
  Let $p\colon T_n \rightarrow S_n$ and $q \colon T_n \rightarrow \EffDiv X$ be the natural projections.
\end{defi}

\begin{lem}\label{lem:irreducible component biprojective}
  Let $X \hookrightarrow {\left(\mathbb{P}^r\right)}^2$ be a closed subscheme defined by bihomogeneous polynomials of total degree $\leq d$. Then
  \[ \#\irr(X) \leq d^{2r}. \]
\end{lem}

\begin{proof}
  Let $L_0$ and $L_1$ be generic hyperplanes of $\mathbb{P}^r$. Then $L_0\times\mathbb{P}^r$ and $\mathbb{P}^r\times L_1$ do not contain any irreducible component of $X$. Let
  \[ U = \left(\mathbb{P}^r \setminus L_0\right)\times\left(\mathbb{P}^r \setminus L_1\right) \simeq \mathbb{A}^{2r}.\]
  Then
  \[\#\irr(X) = \#\irr\left(X \cap U\right),\]
  and $X \cap U \hookrightarrow \mathbb{A}^{2r}$ is defined by polynomials of degree $\leq d$. Consequently, \Cref{cor:conn} proves the inequality.
\end{proof}

\begin{lem}\label{lem:bound Sn}
  Let $n$ be a positive integer. Then $S_n$ is closed in $\left(\mathbb{P}^{N_n}\right)^2$ and
  \[ \# \irr (S_n) \leq {2\max\{n,d\}+(r-1)d \choose r}^{2{n + r \choose r}-2}.  \]
\end{lem}

\begin{proof}
  Take $(f,g) \in (\mathbb{P}^{N_n})^2$ as in \Cref{def:S and T}. Then $\codim_X(\var[X]{f,g}) = 1$, if and only if the intersection of $\var[X]{f,g}$ with $t = \dim X - 1$ number of generic hyperplane sections is nonempty. Let $h_i = \sum_{j=0}^{t-1} \xi_{i,j} x_j$ be a generic hyperplane section for each $i$, where $\xi_{i,j}$ are indeterminates. Take a base extension to $k(\{\xi_{i,j}\}_{i,j})$. Then
  \begin{align*}
    \phantom{\Longleftrightarrow}&\ \codim_X\left(\var[X]{f,g}\right) = 1 \\
    \Longleftrightarrow&\ \var[X]{f,g,h_0\cdots,h_{t-1}} \neq \emptyset \\
    \Longleftrightarrow&\ (x_0,\cdots,x_r)^{2\max\{n,d\} + (r-1)d- r} \not\subset
                         (f,g,h_0\cdots,h_{t-1}) + I_X\\
    \phantom{\Longleftrightarrow}&\ \text{(by \cite[Corollary I.7.4.4.3]{Jan})}\\
    \Longleftrightarrow&\ \rank\left(\left(\left(f,g,h_0\cdots,h_{t-1}\right) + I_X\right)_{2\max\{n,d\} + (r-1)d- r}\right) < {2\max\{n,d\}+(r-1)d \choose r}.
  \end{align*}
  The last condition can be translated into bihomogeneous polynomials in the coefficients of $f$ and $g$ of total degree ${2\max\{n,d\}+(r-1)d \choose r}$. \Cref{lem:irreducible component biprojective} proves the inequality.  
\end{proof}

\begin{lem}
  The set $T_n$ is a closed subset of $S_n \times \EffDiv X$.
\end{lem}

\begin{proof}
  Since $\EffDiv X$ is open and closed in $\Hilb X$, it suffices to show that
  \[ T_n^Q = \left\{ ((f,g),[Z]) \in S_n \times \Hilb_Q X \mid
      Z \subset \var[X]{f,g} \right\} \]
  is closed in $S_n \times \Hilb_Q X$ for every Hilbert polynomial $Q$. Recall that \Cref{thm:explicit construction of Hilb P} gives a closed embedding
  \begin{align*}
    \imath_t\colon\Hilb_Q X &\rightarrow \Gr(P(t),k[x_0,\cdots,x_r]_t) \\
    [Z] &\mapsto \Gamma(\mathscr{I}_Z(t))
  \end{align*}
  for some polynomial $P$ and every large $t$. This gives a closed embedding
  \[ (S_n\times\Hilb_Q X) \hookrightarrow (\mathbb{P}^N\times\Gr(P(t),k[x_0,\cdots,x_r]_t)). \]
  We may assume that $t\geq n$. Notice that $Z \subset V(f,g)$ if and only if the saturation of $(f,g)$ is contained in the saturated ideal defining $Z$. Thus,
  \begin{align*}
    Z \subset V(f,g)
    &\Leftrightarrow (f,g)_t \subset \Gamma(\mathscr{I}_Z(t)) \\
    &\Leftrightarrow \dim\left( \Gamma(\mathscr{I}_Z(t)) + (f,g)_t\right) \leq P(t).
  \end{align*}
  Note that $\mathbb{P}^N\times\Gr(P(t),k[x_0,\cdots,x_r]_t)$ is covered by the standard affine open spaces. In such affine open spaces, the last condition is expressed as $(P(t)+1)\times (P(t)+1)$ minors of some matrix. Consequently, $T_n^Q$ is identified with a closed subset of $\mathbb{P}^N\times\Gr(P(t),k[x_0,\cdots,x_r]_t)$.
\end{proof}

\begin{defi}
  Let $\IrrDiv_n X$ be the union of the irreducible components of $\EffDiv_n X$ which contains at least one closed point corresponding to a reduced and irreducible divisor.
\end{defi}

\begin{lem}\label{lem:comp1}
  Let $F$ be an irreducible component of $\IrrDiv_n X$. Then there is a unique irreducible component $E$ of $T_n$ such that $q(E) = F$.
\end{lem}
\begin{proof}
  \Cref{lem:generic sec} implies that $\EffDiv_n X$ is contained in the image of $q:T_n \rightarrow \EffDiv X$. Because $q$ is proper, there is an irreducible component $E$ of $T_n$ such that $q(E) = F$. Moreover, for any $[D] \in F$,
  \[ q^{-1}([D]) \simeq \left\{ (f,g)\in\left(\mathbb{P}^{N_n}\right)^2 \ \middle|\
      f, g \in \Gamma(\mathscr{I}_D(n)) \right\}. \]
  Then $q^{-1}([D])$ is irreducible, because it is a product of two projective spaces. Thus, such an $E$ is unique.
\end{proof}

\begin{lem}\label{lem:comp2}
  Let $F$ and $E$ be as in \Cref{lem:comp1}. Then there is $(f,g)\in S_n$ such that $E$ is the only irreducible component of $T_n$ satisfying $(f,g) \in p(E)$.
\end{lem}
\begin{proof}
  Let $W \subset \EffDiv_n$ be the complement of the image of the proper morphism
  \newlength{\mylength}\newlength{\mylengthA}\newlength{\mylengthB}
  \settowidth{\mylengthA}{$\times$}\settowidth{\mylengthB}{$,$}
  \setlength{\mylength}{0.5\mylengthA minus 0.5\mylengthB}
  \begin{alignat*}{3}
    \coprod_{t = 1}^{n-1}
    \EffDiv_{t} X &\times \EffDiv_{n - t} X &&\longrightarrow \EffDiv_n X \\
    ([D_0] \kern-\mylength&\kern\mylength,[D_1]) &&\longmapsto [D_0 + D_1].
  \end{alignat*}
  Then $W$ parametrizes the reduced and irreducible divisors of degree $n$ on X. Notice that $W \cap F \neq \emptyset$, because $F \subset \IrrDiv_n X$. The uniqueness of $E$ implies that there is a dense open set $U \subset F$ such that $q^{-1}(U)$ does not intersect with any irreducible component of $T_n$ other than $E$. Thus, we can take $[D] \in W \cap U$. Then $D$ is a reduced and irreducible divisor, since $[D] \in W$. \Cref{lem:generic sec} implies that there is $(f,g) \in S_n$ such that
  \[ p^{-1}((f,g)) = \left\{ ((f,g),[D]) \right\}. \]
  Then $E$ is the only irreducible component of $T_n$ containing $((f,g),[D])$, because $[D] \in U$.
\end{proof}

\begin{lem}\label{lem:comp3}
  Let $F$ and $E$ be as in \Cref{lem:comp1}. Then $p(E)$ is an irreducible component of $S_n$.
\end{lem}
\begin{proof}
    Let
  \[ R_n = q^{-1} (\EffDiv_1 X \cup \EffDiv_2 X \cup \cdots
    \cup \EffDiv_{n^2 \deg X} X) \subset T_n.\]
  Since $R_n$ is proper, the restriction $p|_{R_n}: R_n \rightarrow S_n$ is also proper. Moreover, $p|_{R_n}$ is surjective, because $\deg V_X(f,g) \leq n^2 \deg X$ for every $(f,g) \in S_n$. Thus, every irreducible component of $S_n$ is the image of an irreducible component of $R_n$ under $p|_{R_n}$. Let $B \subset S_n$ be the irreducible component containing $p(E)$. Then \Cref{lem:comp2} implies that $p(E) = B$.
\end{proof}

\begin{lem}\label{lem:IrrDiv bound}
  If $n$ is a positive integer, then
  \[ \# \irr (\IrrDiv_n X) \leq \# \irr (S_n).\]
\end{lem}
\begin{proof}
  \Cref{lem:comp1} and \Cref{lem:comp3} define a map
  \begin{align*}
    \irr(\IrrDiv_n X) &\longrightarrow \irr(S_n)\\
    F &\longmapsto P(E),
  \end{align*}
  where $E$ is determined by $F$ as in \Cref{lem:comp1}. Then \Cref{lem:comp1} and \Cref{lem:comp2} implies that this map is injective.
\end{proof}

\begin{lem}\label{lem:effdiv bound}
  Let $n$ be a positive integer. Then
  \[ \# \irr (\EffDiv_n X) \leq 2^n {2\max\{n,d\}+(r-1)d \choose r}^{2 {n + r \choose r}-2}.  \]
\end{lem}
\begin{proof}
  Notice that
  \begin{alignat*}{3}
    \coprod_{n_0 + \cdots + n_{t-1} = n} \IrrDiv_{n_0} X \times &\cdots \times \IrrDiv_{n_{t-1}} X
    &&\longrightarrow \EffDiv_n X \\
    ([D_0],&\cdots,[D_{t-1}])
    &&\longmapsto [D_0+\cdots+D_{t-1}].
  \end{alignat*}
  is surjective, where the disjoint union runs over all integer partitions of $n$. Thus, the left-hand side has more irreducible components. Take any integer partition $n_0 + \cdots + n_{t-1} = n$. Then by \Cref{lem:bound Sn} and \Cref{lem:IrrDiv bound},
  \begin{align*}
    & \#\irr(\IrrDiv_{n_0} X \times \cdots \times \IrrDiv_{n_{t-1}} X ) \\
    =& \#\irr(\IrrDiv_{n_0} X) \times \cdots \times \#\irr(\IrrDiv_{n_{t-1}} X ) \\
    \leq& {2\max\{n_0,d\}+(r-1)d \choose r}^{2{n_0 + r \choose r}-2} \times \cdots \times {2\max\{n_{t-1},d\}+(r-1)d \choose r}^{2{n_{t-1} + r \choose r}-2}\\
    \leq& {2\max\{n,d\}+(r-1)d \choose r}^{\left({2{n_0 + r \choose r}}-2\right) + \cdots +\left({2{n_{t-1} + r \choose r}}-2\right)}.
  \end{align*}
  Let
  \[ D(x) = 2{x + r \choose r}-2. \]
  Then $D(0) = 0$ and $D$ is concave above in $[0,\infty)$. Therefore,
  \[ D(n_0) + D(n_1) + \cdots + D(n_{t-1}) \leq D(n). \]
  Furthermore, the number of integer partitions of $n$ is less than or equal to $2^n$. This proves the inequality.
\end{proof}

Now, we are ready to give a new upper bound:

\NSbound
\begin{proof}
  If $X$ is a curve or a projective space, then $(\NS X)_\tor = 0$. Thus, we may assume that $r \geq 3$, $d \geq 2$, $\deg X \geq 2$ and $r-2 \geq \codim X \geq 1$. Let $n = (d-1) \codim X \cdot \deg X$. Then
  \[ d \leq n \leq (r-2)(d-1)d^{r-2}. \]
  \Cref{thm:bound by EffDiv} and \Cref{lem:effdiv bound} implies that
  \[ \#(\NS X)_\tor \leq \conn\left(\EffDiv_n X \right) \leq \irr\left(\EffDiv_n X \right) \leq 2^n {2n+(r-1)d \choose r}^{2 {n + r \choose r}-2}. \]
Since
  \begin{align*}
    \log_2 {2n+(r-1)d \choose r}
    &\leq \log_2  \left(2n+(r-1)d\right)^r \\
    &\leq r \log_2  \left(rd^r\right) \\
    &\leq r(\log_2 r + r \log_2 d) \\
    &\leq r^2(1+\log_2 d) \\
    &\leq {r^2 d}
      \intertext{and}
    2 {n + r \choose r} &\leq \frac{2 (n+r)^r}{r!} \\
                        &\leq\frac{ (rd^{r-1})^r}{3},
  \end{align*}
we have
  \begin{align*}
    \log_d\log_2 \left(\#(\NS X)_\tor\right)
    &\leq \log_d\log_2 \left(2^n {2n+(r-1)d \choose r}^{2 {n + r \choose r}}\right) \\
    &\leq \log_d \left(n + r^2 d \frac{ (rd^{r-1})^r}{3}\right)\\
    &\leq \log_d \left( r^{r+2}d^{r(r-1)+1} \right)\\
    &\leq (r+2)\log_2 r+r(r-1)+1 \\
    &\leq 2r\log_2 r + r^2. \qedhere
  \end{align*}
\end{proof}

We now give an application to the torsion subgroups of second cohomology groups.
\begin{lem}\label{thm:NS coh}
  Let $X \hookrightarrow \mathbb{P}^r$ be a smooth projective variety of degree $d$. Let $\ell \neq \charac k$ be a prime number. The embedding $(\NS X)\otimes\mathbb{Z}_\ell \hookrightarrow H^2_\et(X,\mathbb{Z}_\ell)$ \cite[Remark V.3.29]{Mil2} induces by the Kummer sequence restricts to an isomorphism.
  \[ (\NS X)[\ell^\infty] \simeq H^2_\et(X,\mathbb{Z}_\ell)_\tor. \]
\end{lem}
\begin{proof}
See \cite[2.2]{SZ}.
\end{proof}

Therefore, the bound of \Cref{thm:NS better bound} implies the corollaries below.

\SHbound
\begin{proof}
  This follows from \Cref{thm:NS better bound} and \Cref{thm:NS coh}.
\end{proof}

\begin{cor}
  Let $X \hookrightarrow \mathbb{P}^r$ be a smooth projective variety over $\mathbb{C}$ defined by homogeneous polynomials of degree $\leq d$. Then
  \begin{equation*}
    \# H^2_\sing(X^\an,\mathbb{Z})_\tor \leq 2^{d^{r^2+2r\log_2 r}}.
  \end{equation*}
\end{cor}
\begin{proof}
  This follows from \Cref{thm:H2 bound} and the fact that
  \[\prod_{\ell \textup{ is prime}} H^2_\et(X,\mathbb{Z}_\ell)_\tor
    \simeq H^2_\sing(X^\an,\mathbb{Z})_\tor.
  \qedhere\]
\end{proof}

\section{The Number of Generators of \texorpdfstring{$(\NS X)_{\tor}$}{(NS X)tor}}\label{sec:generators}
In this section, $X \hookrightarrow \mathbb{P}^r$ is a smooth connected projective variety, and $\ell$ is a prime number not equal to $\charac k$. The goal is to give a uniform upper bound on the number of generators of $(\NS X)[\ell^\infty]$. Our approach is a simplification of a brief sketch suggested by J\'anos Koll\'ar.

\begin{lem}\label{thm:NS fun}
  Let $X \hookrightarrow \mathbb{P}^r$ be a smooth connected projective variety of degree $d$, and $x_0$ be a geometric point of $X$. Let $\ell \neq \charac k$ be a prime number. Then \[(\NS X)[\ell^\infty]^* \simeq \pi^{\et}_1(X,x_0)^\ab[\ell^\infty].\]
\end{lem}
\begin{proof}
  The exact sequence in \cite[Proposition 69]{Jak} gives an exact sequence
  \[ 0 \rightarrow (\NS X)[\ell^\infty]^* \rightarrow \pi^{\et}_1(X,x_0)^{(\ell)}
    \rightarrow \pi^{\et}_1(\Alb X,0)^{(\ell)} \rightarrow 0, \]
  by taking the maximal pro-$\ell$ abelian quotient. Since $\pi^{\et}_1(\Alb X,0)^{(\ell)}$ is a free $\mathbb{Z}_\ell$-module,
  \[(\NS X)[\ell^\infty]^* \simeq \pi^{\et}_1(X,x_0)^{(\ell)}_\tor
    \simeq \pi^{\et}_1(X,x_0)^{\ab}[\ell^\infty]. \qedhere\]
\end{proof}

If $M$ is a topological manifold, the linking form implies that
$H^2_\sing(M,\mathbb{Z})_\tor^* \simeq H^\sing_1(M,\mathbb{Z})_\tor.$
\Cref{thm:NS coh} and \Cref{thm:NS fun} imply the \'etale analogy that
$H^2_\et(X,\mathbb{Z}_\ell)_\tor^* \simeq\pi^{\et}_1(X,x_0)^\ab[\ell^\infty].$



\begin{lem}\label{thm:pos bound}
  Let $X \hookrightarrow \mathbb{P}^r$ be a smooth connected projective variety of degree $d$. Then $(\NS X)[\ell^\infty]$ is generated by less than or equal to $(d-1)(d-2)$ elements.
\end{lem}
\begin{proof}
    For a general linear space $L \subset \mathbb{P}^r$ of dimension $r - \dim X+1$, the intersection $C \coloneqq X \cap L$ is a connected smooth curve of degree $d$. Take a geometric point $x_0 \in C$. Then the natural map
  \[ \pi_1^{\et}(C,x_0) \rightarrow \pi_1^{\et}(X,x_0) \]
  is surjective, by repeatedly applying the Lefschetz hyperplane theorem for \'etale fundamental groups \cite[XII. Corollaire 3.5]{SGA2}. Let $g$ be the genus of $C$. Then \Cref{thm:NS fun} implies that $(\NS X)[\ell^\infty]^*$ is isomorphic to a subquotient of
  \[ \pi_1^{\et}(C,x_0)^{(\ell)} \simeq \mathbb{Z}_\ell^{2 g}. \]
   Notice that $2g \leq (d-1)(d-2)$ because $\deg C = d$. As a result, $(\NS X)[\ell^\infty]$ is generated by $\leq (d-1)(d-2)$ elements.
 \end{proof}

\NSgen
\begin{proof}
  Since a product of finitely many finite cyclic groups of coprime orders is again cyclic,
  \[N \simeq
    \prod_{\substack{\ell \neq p\\ \ell \textup{ is prime}}} (\NS X)[\ell^\infty] \]
  is a product of $\leq(d-1)(d-2)$ cyclic groups by \Cref{thm:pos bound}.
\end{proof}

\SHgen
\begin{proof}
  This follows from \Cref{thm:pos bound} and \Cref{thm:NS coh}.
\end{proof}

\begin{cor}
  Let $X \hookrightarrow \mathbb{P}^r$ be a smooth connected projective variety of degree $d$ over $\mathbb{C}$. Then $H^2_\sing(X^\an,\mathbb{Z})$ is generated by less than or equal to $(d-1)(d-2)$ elements.
\end{cor}
\begin{proof}
  This follows from \Cref{thm:NSgen} and the fact that
  \[ (\NS X)_\tor
    \simeq \prod_{\ell \textup{ is prime}}(\NS X)[\ell^\infty]
    \simeq \prod_{\ell \textup{ is prime}} H^2_\et(X,\mathbb{Z}_\ell)_\tor
    \simeq H^2_\sing(X^\an,\mathbb{Z})_\tor.
  \qedhere\]
\end{proof}

The bounds in this section exclude the case $\ell = \charac k$, because of the bad behavior of the \'etale cohomology at $p$. One may try to overcome this by using Nori's fundamental group scheme \cite{Nor}. However, the Lefschetz hyperplane theorem for Nori's fundamental group scheme is no longer true \cite[Remark 2.4]{BH}.
\begin{qst}
  Can one use another cohomology to prove the analogue of \Cref{thm:pos bound} for $(\NS X)[p^\infty]$ in characteristic $p$?
\end{qst}

\section*{Acknowledgement}
The author thanks his advisor, Bjorn Poonen, for suggesting the problem, insightful conversation and careful guidance. The author thanks J\'anos Koll\'ar, whose comments led to \Cref{sec:chow} and \Cref{sec:generators}. The author thanks Wei Zhang for helpful conversation regarding \'etale homology. The author thanks Chenyang Xu for helpful conversation regarding Chow varieties. 

\bibliographystyle{plain}
\bibliography{mybib}

\end{document}